\documentclass[10pt, a4paper]{amsart}
\usepackage{amssymb, amsmath, amsfonts, mathrsfs, amsthm}
\usepackage[utf8]{inputenc}
\usepackage{lmodern}
\usepackage[T1]{fontenc}
\usepackage[a4paper,includeheadfoot,margin=2.54cm]{geometry}
\usepackage[matrix, arrow, curve]{xy}

\DeclareMathOperator{\OO}{\mathscr{O}}
\DeclareMathOperator{\PP}{\mathbb{P}}

\DeclareMathOperator{\QQ}{\mathbb{Q}}

\DeclareMathOperator{\Bim}{\mathrm{Bim}}
\DeclareMathOperator{\Aut}{\mathrm{Aut}}
\DeclareMathOperator{\Psaut}{\mathrm{PsAut}}

\numberwithin{equation}{section}

\begin{document} 

\title{Jordan property for groups of bimeromorphic automorphisms of compact K\"ahler threefolds}
\date{}
\author{Aleksei Golota}
\thanks{This work is supported by the Russian Science Foundation under grant 18-11-00121.}

\newtheorem{theorem}{Theorem}[section] 
\newtheorem*{ttheorem}{Theorem}
\newtheorem{lemma}[theorem]{Lemma}
\newtheorem{proposition}[theorem]{Proposition}
\newtheorem*{conjecture}{Conjecture}
\newtheorem{corollary}[theorem]{Corollary}

{\theoremstyle{remark}
\newtheorem{remark}[theorem]{Remark}
\newtheorem{example}[theorem]{Example}
\newtheorem{notation}[theorem]{Notation}
\newtheorem{question}{Question}
}

\theoremstyle{definition}
\newtheorem{construction}[theorem]{Construction}
\newtheorem{definition}[theorem]{Definition}

\begin{abstract} Let $X$ be a non-uniruled compact K\"ahler space of dimension 3. We show that the group of bimeromorphic automorphisms of $X$ is Jordan. More generally, the same result holds for any compact K\"ahler space admitting a quasi-minimal model.
\end{abstract}

\maketitle

\section{Introduction} A fruitful way to study various groups of geometric nature is to explore their finite subgroups. For example, we can consider groups of automorphisms of algebraic (e.g. projective) varieties. More challenging are groups of birational automorphisms, the case of Cremona group $\mathrm{Cr}_n(\mathbb{C}) = \mathrm{Bir}(\mathbb{P}^n_{\mathbb{C}})$ being particularly famous (see e.g. \cite{D\'es21} and references therein for more details). 

A well-known theorem of Jordan proved in \cite{Jor78} describes finite subgroups in the general linear group. It says that there exists a function $J \colon \mathbb{N} \to \mathbb{N}$ such that all finite subgroups in $\mathrm{GL}_n(\mathbb{C})$ can be obtained as extensions of groups of order at most $J(n)$ by finite abelian groups. Popov in \cite{Pop11} suggested the name ``Jordan groups'' for groups with the above property.

When $X$ is a projective variety of any dimension over a field $\Bbbk$ of $\mathrm{char}(\Bbbk) = 0$ the Jordan property for $\Aut(X)$ was established by Meng and Zhang in \cite{MZ18}. This result was extended to automorphism groups of compact K\"ahler spaces by Kim \cite{Kim18} and automorphism groups of compact complex spaces in Fujiki's class $\mathcal{C}$ by Meng, Perroni and Zhang \cite{MPZ20}. Also, Popov \cite{Pop18} proved that real Lie groups have the Jordan property; in particular, for any compact complex space $X$ the neutral component of $\Aut(X)$ is Jordan.

As for groups of birational automorphisms, the Jordan property for $\mathrm{Cr}_2(\mathbb{C}) = \mathrm{Bir}(\mathbb{P}^2_{\mathbb{C}})$ was established by Serre in \cite{Ser07}. For complex projective surfaces Popov in \cite{Pop11} verified the Jordan property in all cases except $\PP^1 \times E$ where $E$ is an elliptic curve. Later, Zarhin \cite{Zar14} showed that $\mathrm{Bir}(\PP^1 \times E)$ is not Jordan and provided similar examples in higher dimensions. Important results on Jordan property of $\mathrm{Bir}(X)$ for $X$ projective over a field of characteristic 0 were obtained by Prokhorov and Shramov. They proved that the Jordan property holds for $\mathrm{Bir}(X)$ for $X$ non-uniruled \cite{PS14} and for $X$ rationally connected \cite{PS16}.

For compact complex manifolds it is natural to consider the group $\Bim(X)$ of bimeromorphic maps from $X$ to itself. Prokhorov and Shramov \cite{PS21a} settled the case of compact complex surfaces. They proved that that automorphism groups of all compact complex surfaces are Jordan and that bimeromorphic automorphism groups of all compact complex surfaces except those bimeromorphic to $\PP^1 \times E$ are Jordan as well.

 In a series of papers \cite{PS18, PS20, PS21b} Prokhorov and Shramov studied groups of bimeromorphic selfmaps of compact K\"ahler spaces of dimension 3. In the uniruled case, they proved that $\Bim(X)$ is Jordan unless $X$ is bimeromorphic to a space from one of finitely many explicitly described families, see  \cite[Theorems 1.3 and 1.4]{PS20}. For groups of bimeromorphic automorphisms of non-uniruled K\"ahler spaces they were able to show the Jordan property under an additional assumption. Recall that the Kodaira dimension $\kappa(X)$ of a compact K\"ahler space $X$ is defined to be the Kodaira dimension of any smooth manifold $X'$ bimeromorphic to $X$. Analogously, the irregularity of $X$ is defined to be $q(X) = H^1(X', \OO_{X'})$ for any smooth manifold $X'$ bimeromorphic to $X$.

\begin{theorem}{\cite[Theorem 1.3]{PS21b}}\label{PStheorem}
Let $X$ be a compact K\"ahler space of dimension ~3. Assume that $\kappa(X) \geqslant 0$ and $q(X) > 0$. Then the group $\Bim(X)$ is Jordan.
\end{theorem}

The purpose of this note is to prove the following theorem.

\begin{theorem}\label{MainThm3} The group $\Bim(X)$ is Jordan for any non-uniruled compact K\"ahler space $X$ of dimension ~3.
\end{theorem}

To achieve this we generalize the arguments from \cite[Corollary 3.3]{Fuj81} and \cite[Proposition 4.5]{PS21b} to singular K\"ahler spaces, see Theorems ~\ref{Fujiki} and ~\ref{MainThm}, respectively. Together with existence of minimal models for $\QQ$-factorial terminal K\"ahler spaces of dimension 3 from \cite{HP16} this gives the desired conclusion of Theorem \ref{MainThm3}. In higher dimensions, we show that the same result holds assuming existence of a quasi-minimal model of $X$ (Theorem \ref{MainThm2}).

\textbf{Acknowledgement.} The author thanks Constantin Shramov for suggesting this problem and Yuri Prokhorov for valuable discussions.

\section{Generalities on the Jordan property} The following definition was suggested by Popov in \cite{Pop11}.

\begin{definition} \label{Jordan} Let $G$ be a group. We say that $G$ is {\em Jordan} (or has the {\em Jordan property}) if there is a constant $J(G) \in \mathbb{N}$ such that for any finite subgroup $H \subset G$ there is a normal abelian subgroup $A \subset H$ of index at most $J(G)$. 
\end{definition}

A classical result of C. Jordan says that this property holds for $G = \mathrm{GL}_n(\mathbb{C})$ (for a proof see e.g. \cite[Theorem 36.13]{CR62}). Clearly, a subgroup of a Jordan group is again Jordan; therefore, all linear algebraic groups over $\mathbb{C}$ are Jordan.

On the other hand, quotient groups and extensions of Jordan groups are not necessarily Jordan. To overcome this difficulty we also consider a more restrictive class of groups with bounded finite subgroups. Recall that a group $G$ has {\em bounded finite subgroups} if there is a number $B(G)$ such that any finite subgroup $H \subset G$ we have $|H| \leqslant B(G)$ (\cite[Definition 2.7]{Pop11}). An easy but important result below (see \cite[Lemma 2.9]{Pop11}) says that an extension of a group with bounded finite subgroups by a Jordan group remains Jordan.

\begin{proposition}\label{Extensions} Consider an exact sequence of groups $$ 1 \to G_1 \to G_2 \to G_3 \to 1.$$ Suppose that $G_1$ is Jordan and $G_3$ has bounded finite subgroups. Then $G_2$ is Jordan.
\end{proposition}

Examples of groups with bounded finite subgroups are given by a classical theorem of Minkowski (see e. g. \cite[Theorem 1]{Ser07}). 

\begin{theorem}\label{Minkowski} The orders of finite subgroups in $G = \mathrm{GL}_n(\mathbb{Q})$ are bounded by a natural number $M(n)$ depending on $n$ only.
\end{theorem}

We will need the following result of Kim \cite[Theorem 1.1]{Kim18} (for the notion of a singular K\"ahler space see Definition \ref{Kahler} below).

\begin{theorem}\label{Kim} Let $X$ be a normal compact K\"ahler space. Then the group $\Aut(X)$ of biholomorphic automorphisms of $X$ is Jordan.
\end{theorem}

\section{Singular compact K\"ahler spaces and their minimal models}

We recall basic definitions related to singular K\"ahler spaces mostly following \cite{HP16}.

\begin{definition} \label{Kahler} Let $X$ be an irreducible and reduced complex space; denote the subsets of its singular and non-singular points by $X_{\mathrm{sing}}$ and $X_{\mathrm{ns}}$, respectively. A K\"ahler form on $X$ is a closed positive real $(1,1)$-form $\omega$ satisfying the following condition of {\em existence of local potentials}: for any $x \in X_{\mathrm{sing}}$ there exists an open neighborhood $x \in U \subset X$ with a closed embedding $i_U \colon U \subset V$ into an open subset $V \subset \mathbb{C}^N$ such that $$\omega|_{U \cap X_{\mathrm{ns}}} = i\partial\bar\partial f|_{U \cap X_{\mathrm{ns}}}$$ for a smooth strictly plurisubharmonic function $f \colon V \to \mathbb{C}$. An irreducible and reduced complex space $X$ is {\em K\"ahler} if there exists a K\"ahler form on $X$.
\end{definition}

\begin{remark}\label{Singularities} We restrict ourselves to singular K\"ahler spaces that are normal and have rational singularities. The latter condition ensures that the space $N^1(X)$ defined via $\partial\bar\partial$-cohomology embeds into $H^2(X, \mathbb{R})$ \cite[Remark 3.7]{HP16}. 
\end{remark}

The definition of terminal singularities for singular K\"ahler spaces is identical to the projective case. Every terminal singularity is rational by \cite{Elk81} (see also \cite[Theorem 5.22]{KM98}).

To define minimal and quasi-minimal models of compact K\"ahler spaces we need to introduce notions of nefness and ``nefness in codimension one'' in non-projective context (see \cite{Bou04, HP16} for more details).

\begin{definition}\label{Nef} Let $X$ be a normal compact K\"ahler space with rational singularities. We say that a class $\alpha \in H^{1,1}(X, \mathbb{R})$ is {\em nef} if it belongs to the closure of the cone of K\"ahler classes. Also, a class $\alpha$ is called {\em semipositive} if there exists a smooth semipositive form representing $\alpha$.
\end{definition}

\begin{definition}\label{ModNef} Let $X$ be a normal compact K\"ahler space with rational singularities. We say that a class $\alpha \in H^{1,1}(X, \mathbb{R})$ is {\em modified nef} if it belongs to the closure of the cone generated by classes of the form $\mu_*\omega$ where $\mu \colon Y \to X$ is an arbitrary bimeromorphic morphism from a smooth K\"ahler manifold $Y$ and $\omega$ is a K\"ahler class on $Y$.
\end{definition}

Recall that a class $\alpha \in H^{1,1}(X, \mathbb{R})$ is {\em pseudoeffective} if it can be represented by a closed positive $(1,1)$-current with local potentials. For any pseudoeffective class on a smooth manifold $X$ Boucksom defined a so-called {\em divisorial Zariski decomposition} \cite[Definition 3.7]{Bou04}.

\begin{proposition}\label{DivZar} Let $X$ be a compact K\"ahler manifold and let $\alpha \in H^{1,1}(X, \mathbb{R})$ be a pseudoeffective class. Then there exists a {\em divisorial Zariski decomposition} $$\alpha = P(\alpha) + N(\alpha)$$ such that \begin{itemize} \item The {\em positive part} $P(\alpha)$ is a modified nef class; \item The {\em negative part} $N(\alpha) = \sum_ia_i[D_i]$ is a class of an effective $\mathbb{R}$-divisor, which is exceptional in the sense that the cone generated by the classes $[D_i]$ intersects the cone of modified nef classes at 0 only.
\end{itemize}
\end{proposition}

\begin{remark}\label{ZarSing} For $X$ a normal compact K\"ahler space with rational singularities we can still define divisorial Zariski decomposition of a pseudoeffective class $\alpha \in H^{1,1}(X, \mathbb{R})$. To do so, we take a resolution of singularities $\mu \colon Y \to X$ and apply Proposition \ref{DivZar} to $\mu^*\alpha$ to obtain a decomposition $$\mu^*\alpha = P(\mu^*\alpha) + N(\mu^*\alpha).$$ Then we can define $P(\alpha) = \mu_*P(\mu^*\alpha)$ and $N(\alpha) = \mu_*N(\mu^*\alpha)$. Since $X$ is nonsingular in codimension one, the decomposition $\alpha = P(\alpha) + N(\alpha)$ satisfies the conditions in Proposition \ref{DivZar}. Moreover, it does not depend on a choice of $\mu \colon Y \to X$ since any two resolutions of $X$ can be dominated by a common resolution.
\end{remark}

We define minimal models of singular compact K\"ahler spaces following the conventions in \cite{HP16}. Note that the definition of $\mathbb{Q}$-factorial singularities in \cite{HP16} is as follows: every Weil divisor is $\mathbb{Q}$-Cartier and some reflexive power of the dualizing sheaf $K_X$ is a line bundle. 

\begin{definition} A compact K\"ahler space $X$ with terminal $\mathbb{Q}$-factorial singularities is called \begin{itemize} \item {\em minimal} (or a {\em minimal model}) if the canonical class $K_X$ is nef; \item {\em quasi-minimal} (or a {\em quasi-minimal model}) if $K_X$ is modified nef. \end{itemize}  
\end{definition}

Let $X$ be a compact K\"ahler space of dimension 3 with terminal $\mathbb{Q}$-factorial singularities. Suppose that $X$ is not uniruled; then the canonical class $K_X$ is pseudoeffective. This follows from \cite[Theorem ~2.6]{BDPP13} for $X$ projective in any dimension and from \cite{Bru06} for $X$ non-projective in dimension 3. In \cite[Theorem ~1.1]{HP16} H\"oring and Peternell proved existence of minimal models for compact K\"ahler spaces of dimension 3 with terminal $\mathbb{Q}$-factorial singularities and $K_X$ pseudoeffective. Thus, any non-uniruled compact K\"ahler space has a minimal model. 

\begin{theorem} \label{HP} Let $X$ be a terminal $\QQ$-factorial compact K\"ahler space of dimension 3. Suppose that $X$ is not uniruled. Then there exists a sequence of bimeromorphic maps $\varphi \colon X \dasharrow X'$ such that $X'$ is a minimal model.
\end{theorem}

Note that any minimal model is also a quasi-minimal model. When $X$ is projective, our definition of quasi-minimal models is equivalent to \cite[Definition 4.2]{PS14}. Existence of quasi-minimal models for projective non-uniruled varieties with terminal $\mathbb{Q}$-factorial singularities was shown in \cite[Lemma 4.4]{PS14}.

\section{Bimeromorphic maps of quasi-minimal models}

In this section we consider bimeromorphic maps of normal compact K\"ahler spaces with rational singularities and discuss sufficient conditions for such maps to be holomorphic. Recall that for any bimeromorphic map $f \colon X \dasharrow X'$ of reduced complex spaces there exists a resolution of indeterminacies: 
\begin{equation}
\label{Diagram}
\xymatrix{
& Y \ar[dl]_{p} \ar[dr]^{q}\\
X \ar@{-->}[rr]^{f} && X'
}
\end{equation}
Here $Y$ is a smooth complex manifold and $p$ and $q$ are bimeromorphic morphisms.

The following fact is well-known to experts; we include a proof for reader's convenience.

\begin{lemma}\label{Isolated} Let $f \colon X \dasharrow X'$ be a bimeromorphic map of reduced complex spaces such that $X'$ is normal. Suppose that there exists a resolution \eqref{Diagram} such that every fiber of $q$ is mapped by $p$ to a point. Then the inverse map $f^{-1}$ is holomorphic.
\end{lemma}

\begin{proof} Suppose that there is a point $x$ in the indeterminacy set $\mathrm{Ind}(f^{-1})$. By assumption we can extend the map $f^{-1}$ to a continuous (but a priori not holomorphic) map $X' \to X$. There is an open subset $U \subset X$ containing $f^{-1}(x)$ such that $U$ embeds as an open subset into $\mathbb{C}^N$. In other words, in some neighborhood $V \subset X'$ the map $f^{-1}$ is given by an $N$-tuple of functions which are holomorphic outside a subset $V \cap \mathrm{Ind}(f^{-1})$ of codimension at least 2 in $V$. Since $X'$ (and thus $V$) is normal, by the extension theorem \cite[p. 144]{GR84} the map $f^{-1}$ extends holomorphically to $x$.
\end{proof}

The following proposition is due to Hanamura \cite[Lemma 3.4]{Ha87} and Koll\'ar \cite[Lemma 4.3]{Kol89} in the case of minimal algebraic varieties. We prove a slightly more general version using divisorial Zariski decompositions.

\begin{proposition} \label{Kollar} Let $f \colon X \dasharrow X'$ be a bimeromorphic map between compact K\"ahler spaces with terminal singularities. Suppose that $K_X$ and $K_{X'}$ are modified nef. Then $f$ is an isomorphism in codimension one.
\end{proposition}

\begin{proof} We consider a resolution of indeterminacies \eqref{Diagram}. Let us denote by $E_i \subset Y$ (respectively, by $F_j \subset Y$) irreducible $p$-exceptional (respectively, $q$-exceptional) divisors. We have $$p^*K_{X} + \sum_ia_iE_i = K_Y = q^*K_{X'} + \sum_jb_jF_j$$ where all coefficients $a_i, b_j$ are strictly positive by terminality of $X$ and $X'$. In particular, $K_Y$ is pseudoeffective and therefore by Proposition  \ref{DivZar} it has a divisorial Zariski decomposition: $$K_Y = P(K_Y) + N(K_Y).$$ Since $p_*K_Y = K_X$ and $q_*K_Y = K_{X'}$ are modified nef, any irreducible component of the negative part $N(K_Y)$ is both $p$- and $q$-exceptional. On the other hand, exceptional divisors $E = \sum_ia_iE_i$ and $F = \sum_jb_jF_j$ are also exceptional in the sense of Proposition \ref{DivZar} and thus $E_i$ and $F_j$ for all $i, j$ are components of the negative part $N(K_Y)$. Therefore the map $f$ gives an isomorphism between open subsets $X\setminus p(\cup_{i,j}E_i \cup F_j)$ and $X'\setminus q(\cup_{i,j}E_i \cup F_j)$ whose complements have codimension at least 2.
\end{proof}

Let $f \colon X \dasharrow X'$ be a bimeromorphic map and let $\alpha \in H^{1,1}(X, \mathbb{R})$ be a class on $X$. To define the pushforward $f_*\alpha \in H^{1,1}(X', \mathbb{R})$ we take a resolution \eqref{Diagram} and set $f_*\alpha = q_*p^*\alpha$. If $\alpha$ is pseudoeffective then $f_*\alpha$ is pseudoeffective as well. Recall that the singular locus of a closed positive $(1,1)$-current $T$ is a subset defined locally as the set of points where a local plurisubharmonic potential of $T$ is not bounded from below. 

\begin{definition} Let $\alpha \in H^{1,1}(X, \mathbb{R})$ be a pseudoeffective class. The {\em singular locus} $S(\alpha)$ is defined to be the intersection of singular loci of all closed positive $(1,1)$-currents $T \in \alpha$. 
\end{definition}

If the class $\alpha$ is semipositive (e. g. if $\alpha$ is a K\"ahler class) then clearly $S(\alpha) = \emptyset$. 

We need the following lemma which is proved in \cite[Lemma 2.4]{Fuj81} for $Y$ and $X$ smooth manifolds. However, the proof works when $X$ is singular as well; we reproduce it here to be self-contained. Recall that a closed subset $S \subset X$ is {\em thin} in $X$ if for every point $x \in X$ there exists an open neighborhood $U$ of $x$ such that $S \cap U$ is contained in a nowhere dense analytic subset of $U$.

\begin{lemma}\label{FujikiLemma} Let $f \colon Y \to X$ be a bimeromorphic morphism from a smooth K\"ahler manifold $Y$ to a normal compact K\"ahler space $X$ with exceptional set $E$. Let $\alpha \in H^{1,1}(Y, \mathbb{R})$ be a pseudoeffective class. If $E \cap S(\alpha)$ is a thin subset in $E$ then we have the equality $$f^*f_*\alpha - \alpha = \sum_ir_i[E_i]$$ where $E_i$ are irreducible components of $E$ and $r_i$ are some nonnegative real numbers. 
\end{lemma}

\begin{proof} Choose a closed positive current $T$ representing the class $\alpha$. The problem is local on $X$, so we may replace $X$ by an open neighborhood $U \subset \mathbb{C}^N$ such that $f_*T = i\partial\bar\partial\varphi$ for a plurisubharmonic function $\varphi$ on $U$. Then the class $f^*f_*\alpha$ is represented by a closed positive current $i\partial\bar\partial (\varphi \circ f)$. Let $y \in Y$ be an arbitrary point. We can choose an open neighborhood $V$ of $y$ such that on $V$ we have $T = i\partial\bar\partial \psi$ for a plurisubharmonic function $\psi$. Then the function $\varphi \circ f - \psi$ is pluriharmonic on $V \setminus (V \cap E)$. Since $Y$ is smooth and $S(\alpha) \cap E$ is thin in $E$ by assumption, we can extend the function $\varphi \circ f - \psi$ to a plurisubharmonic function on $V$ (see \cite[Lemma 2.6]{Fuj81}). We apply \cite[Proposition 3.1.3]{King71} (see also \cite[pp. 744-745]{Fuj81}) to get the following equality of currents on $V$: $$i\partial\bar\partial(\varphi \circ f) - i\partial\bar\partial\psi = \sum_ir_i(V \cap E_i)$$ where $r_i$ are nonnegative real numbers. Since the point $y \in Y$ was arbitrary, we obtain the desired equality of $(1, 1)$-classes. 
\end{proof}

The next step in our proof is a generalization of a result proved by Fujiki \cite[Corollary 3.3]{Fuj81} for smooth manifolds.

\begin{theorem} \label{Fujiki}
Let $f \colon X \dasharrow X'$ be a bimeromorphic map of normal compact K\"ahler spaces with rational singularities. Suppose that $f$ does not contract divisors. If there exists a K\"ahler class $\alpha \in H^{1,1}(X, \mathbb{R})$ such that the class $\alpha' = f_*\alpha$ is semipositive then the inverse map $f^{-1}$ is holomorphic.
\end{theorem}

\begin{proof} Again, we consider a resolution of indeterminacies \eqref{Diagram}. We may assume that $q$ is projective. Moreover, we may assume that $\mathrm{Exc}(p) = \sum_iE_i$ and $\mathrm{Exc}(q) = \sum_jF_j$ are divisors.
By definition of the pushforward we have $\alpha' = f_*\alpha = q_*p^*\alpha$; let us consider the class $$q^*q_*p^*\alpha - p^*\alpha \in H^{1,1}(Y, \mathbb{R}).$$ Since the class $p^*\alpha$ is represented by a smooth semipositive form, we have $S(\alpha) = \emptyset$ so we can apply Lemma \ref{FujikiLemma} and obtain 
\begin{equation}\label{alpha1}
q^*q_*p^*\alpha - p^*\alpha = q^*\alpha' - p^*\alpha = \sum_jr_j[F_j]
\end{equation} for some $r_j \geqslant 0$. Since $\alpha'$ is semipositive by assumption, the same is true for the class $q^*\alpha'$, therefore $S(q^*\alpha') = \emptyset$. Note that by assumption $f$ does not contract divisors, so every $q$-exceptional divisor is also $p$-exceptional. Therefore from \eqref{alpha1} we have $$p_*q^*\alpha' = p_*(p^*\alpha + \sum_jr_j[F_j]) = \alpha + \sum_jr_jp_*[F_j] = \alpha.$$ We apply Lemma \ref{FujikiLemma} to the map $p$ and obtain the equality \begin{equation}\label{alpha2}
p^*p_*q^*\alpha' - q^*\alpha' = p^*\alpha - q^*\alpha' = \sum_is_i[E_i]
\end{equation} for some $s_i \geqslant 0$. Adding the expressions \eqref{alpha1} and \eqref{alpha2} we finally obtain the equality
\begin{equation}
\label{Equality}
p^*\alpha = q^*\alpha'.
\end{equation}

If $q$ is an isomorphism then $f^{-1}$ is clearly holomorphic. Suppose that $q$ is not an isomorphism. Then by projectivity of $q$ every fiber is projective and therefore is covered by curves. If every irreducible curve in every fiber of $q$ is contracted to a point by $p$, then $f^{-1}$ is holomorphic by Lemma \ref{Isolated}. Thus we may assume that there exists an irreducible curve $\widetilde{C} \subset Y$ contracted to a point by $q$ and such that $p(\widetilde{C})$ is a curve. Write $p_*\widetilde{C} = mC$ for some $m > 0$. By the projection formula and equality \eqref{Equality} we obtain 
$$ 0 < \alpha \cdot mC = \alpha \cdot p_*\widetilde{C} = p^*\alpha \cdot \widetilde{C} = q^*\alpha'\cdot \widetilde{C} = \alpha' \cdot q_*\widetilde{C} = 0,$$ which is a contradiction. Thus the set $\mathrm{Ind}(f^{-1})$ is empty and $f^{-1}$ is holomorphic.
\end{proof}

Theorem \ref{Fujiki} gives a sufficient condition for an isomorphism in codimension one between singular K\"ahler spaces to be biholomorphic in terms of its action on K\"ahler classes.

\begin{corollary}\label{FujikiCor} Let $f \colon X \dasharrow X'$ be a bimeromorphic map of normal compact K\"ahler spaces with rational singularities. Suppose that $f$ is an isomorphism in codimension one. If there exists a K\"ahler class $\alpha \in H^{1,1}(X, \mathbb{R})$ such that the class $\alpha' = f_*\alpha$ is also K\"ahler then $f$ is biholomorphic.
\end{corollary}

\begin{remark} Jia and Meng proved that the conclusion of Theorem \ref{Fujiki} holds if the class $f_*\alpha$ is merely assumed to be nef ( see \cite[Proposition 2.1]{JM22}). In this case our proof of the equality \eqref{Equality} does not work since we cannot guarantee that $S(q^*\alpha') \cap \mathrm{Exc}(p)$ is thin in $\mathrm{Exc}(p)$. Jia and Meng overcome this difficulty by using the negativity lemma (\cite[Claim 2.2]{JM22}).
\end{remark}

\section{The main results}

Now we can state and prove our main results. The first one generalizes \cite[Proposition 4.5]{PS21b} to the singular case. We denote by $\Psaut(X)$ the group of pseudoautomorphisms of $X$, that is, bimeromorphic maps from $X$ to itself which are isomorphisms in codimension one. This group acts on $H^2(X, \mathbb{Q})$ by pushforward $f \to f_* = (f^{-1})^*$; we also denote $$\Psaut(X)_{\tau} = \{f \in \Psaut(X) \mid f_*|_{H^{2}(X, \QQ)} = \mathrm{id} \}.$$

\begin{theorem} \label{MainThm} Let $X$ be a normal compact K\"ahler space with rational singularities. Then the group $\Psaut(X)$ is Jordan. 
\end{theorem}

\begin{proof} The action of $\mathrm{PsAut}(X)$ on $H^{2}(X, \mathbb{Q})$ gives an exact sequence of groups $$1 \to \Psaut(X)_{\tau} \to \Psaut(X) \to \Psaut(X)/\Psaut(X)_{\tau} \to 1.$$ Any element $f \in \mathrm{PsAut}(X)_{\tau}$ acts trivially on the space $H^{2}(X, \mathbb{R}) = H^{2}(X, \QQ) \otimes_{\mathbb{Q}} \mathbb{R}$ and in particular it preserves every K\"ahler class. By Theorem \ref{Fujiki}, the group $\Psaut(X)_{\tau}$ is a subgroup of $\Aut(X)$. Therefore it is Jordan by Theorem \ref{Kim}. Note that the quotient group $\Psaut(X)/\Psaut(X)_{\tau}$ embeds into the linear group $\mathrm{GL}(H^{2}(X, \QQ))$ and therefore by Theorem \ref{Minkowski} it has bounded finite subgroups. Hence by Proposition \ref{Extensions} the group $\Psaut(X)$ is Jordan as well. 
\end{proof}

Assuming that $X$ admits a quasi-minimal model we can obtain the same conclusion for $\Bim(X)$. In particular, this gives another proof of \cite[Theorem 1.8 (ii)]{PS14}.

\begin{theorem} \label{MainThm2} Let $X$ be a compact K\"ahler space admitting a quasi-minimal model. Then the group $\mathrm{Bim}(X)$ is Jordan. 
\end{theorem}

\begin{proof} By assumption, there exists a bimeromorphic map $\varphi \colon X \dasharrow X'$ such that $X'$ is a compact K\"ahler space with terminal $\mathbb{Q}$-factorial singularities and $K_{X'}$ is modified nef. By Proposition \ref{Kollar} we have $$\Bim(X) \simeq \mathrm{Bim}(X') = \mathrm{PsAut}(X')$$ and the latter group is Jordan by Theorem \ref{MainThm}.
\end{proof}

In particular, we can prove Theorem \ref{MainThm3}.

\begin{proof}[Proof of Theorem 1.2] Replace $X$ by a resolution of singularities; then the statement follows immediately from Theorem \ref{MainThm2} and Theorem \ref{HP}.
\end{proof}

\begin{remark}\label{Flops} We can also prove Theorem \ref{MainThm3} using the fact that any bimeromorphic map between minimal models in dimension 3 can be decomposed as a sequence of analytic flops \cite[Theorem 4.9]{Kol89}.
\end{remark}

\begin{remark} Conjecturally, any compact K\"ahler space $X$ which is not uniruled should admit a minimal model (or at least a quasi-minimal model). By Theorem \ref{MainThm2} this would give the Jordan property for $\Bim(X)$ in higher dimensions. In \cite{CH20} Cao and H\"oring made some progress in this direction. 
\end{remark}

\flushleft{Aleksei Golota \\
National Research University Higher School of Economics, Russian Federation \\
Laboratory of Algebraic Geometry, NRU HSE, 6 Usacheva str.,Moscow, Russia, 119048 \\
\texttt{golota.g.a.s@mail.ru, agolota@hse.ru}}


\begin{thebibliography}{XXXXX}

\bibitem[Bou04]{Bou04}
S. Boucksom.
\newblock Divisorial Zariski decompositions on compact complex manifolds.
\newblock {\em Ann. Sci. ENS} (4) 37, no. 1, 45--76 (2004).

\bibitem[BDPP13]{BDPP13}
S. Boucksom, J.-P. Demailly, M. P\u{a}un and T. Peternell.
\newblock The pseudo-effective cone of a compact K\"ahler manifold and varieties of negative Kodaira dimension.
\newblock {\em J. Algebraic Geom.} 22 (2013), no. 2, 201--248.

\bibitem[Bru06]{Bru06}
M. Brunella. 
\newblock A positivity property for foliations on compact K\"ahler manifolds.
\newblock {\em Internat. J. Math.}, 17(1):35--43, 2006.

\bibitem[CH20]{CH20}
J. Cao and A. H\"oring.
\newblock Rational curves on compact K\"ahler manifolds.
\newblock {\em J. Differential Geom.} 114 (1) 1--39, January 2020.

\bibitem[CR62]{CR62}
C. W. Curtis and I. Reiner.
\newblock Representation Theory of Finite Groups and Associative Algebras. 
\newblock Wiley, New York, 1962.

\bibitem[D\'es21]{D\'es21}
J. D\'eserti.
\newblock The Cremona group and its subgroups.
\newblock Mathematical Surveys and Monographs, 252. American Mathematical Society, Providence, RI, 2021. x+187 pp.

\bibitem[Elk81]{Elk81}
R. Elkik. 
\newblock Rationalité des singularités canoniques.
\newblock {\em Inventiones mathematicae} 64, pages 1–6 (1981).

\bibitem[Fuj81]{Fuj81}
A. Fujiki.
\newblock A theorem on bimeromorphic maps of Kähler manifolds and its applications.
\newblock {\em Publ. Res. Inst. Math. Sci.} 17 (1981), no. 2, 735--754. 

\bibitem[GR84]{GR84}
H. Grauert and R. Remmert.
\newblock Coherent Analytic Sheaves.
\newblock Grundl. Math. Wiss. 265, Springer, Berlin, 1984.

\bibitem[Ha87]{Ha87}
M. Hanamura.
\newblock On the birational automorphism groups of algebraic varieties.
\newblock {\em Compositio Mathematica}, Tome 63 (1987) no. 1, pp. 123--142. 

\bibitem[HP16]{HP16}
A. Höring and T. Peternell. 
\newblock Minimal models for Kähler threefolds. 
\newblock {\em Invent. Math.} 203 (2016), no. 1, 217--264.

\bibitem[JM22]{JM22}
J. Jia and S. Meng.
\newblock Moishezon manifolds with no nef and big classes.
\newblock arXiv:2208.12013.

\bibitem[Jor78]{Jor78}
C. Jordan. 
\newblock Mémoire sur les équations différentielles linéaires à intégrale algébrique.
\newblock {\em Journal für die reine und angewandte Mathematik (Crelles Journal)}, vol. 1878, no. 84, 1878, pp. 89--215.

\bibitem[Kim18]{Kim18}
J. H. Kim.
\newblock Jordan property and automorphism groups of normal compact Kähler varieties. 
\newblock {\em Commun. Contemp. Math.} 20 (2018), no. 3, 1750024, 9 pp.

\bibitem[King71]{King71}
J. King.
\newblock The currents defined by analytic varieties.
\newblock {\em Acta Math.} 127: 185-220 (1971).

\bibitem[Kol89]{Kol89}
J. Kollár.
\newblock Flops.
\newblock {\em Nagoya Math. J.} 113 (1989), 15–36.

\bibitem[KM98]{KM98}
J. Kollár, S. Mori.
\newblock Birational geometry of algebraic varieties.
\newblock Cambridge Tracts in Mathematics, vol. 134, Cambridge University Press, Cambridge, 1998. 

\bibitem[MZ18]{MZ18}
S. Meng and D.-Q. Zhang. 
\newblock Jordan property for non-linear algebraic groups and projective varieties. 
\newblock {\em Amer. J. Math.} 140 (2018), no. 4, 1133–1145.

\bibitem[MPZ20]{MPZ20}
S. Meng, F. Perroni and D.-Q. Zhang.
\newblock Jordan property for automorphism groups of compact spaces
in Fujiki’s class $\mathcal{C}$.
\newblock arXiv:2011.09381.

\bibitem[Pop11]{Pop11}
V. L. Popov. 
\newblock On the Makar-Limanov, Derksen invariants, and finite automorphism groups of algebraic
varieties. 
\newblock In {\em Peter Russell’s Festschrift, Proceedings of the conference on Affine Algebraic Geometry
held in Professor Russell’s honour, 1–5 June 2009, McGill Univ., Montreal}., volume 54 of CRM Proc. and Lect. Notes, pages 289–311, 2011.

\bibitem[Pop18]{Pop18}
V. L. Popov.
\newblock The Jordan Property for Lie Groups
and Automorphism Groups of Complex Spaces.
\newblock {\em Mathematical Notes}, 2018, Vol. 103, No. 5, pp. 811–819. 

\bibitem[PS14]{PS14}
Yu. Prokhorov and C. Shramov.
\newblock Jordan property for groups of birational selfmaps.
\newblock {\em Compositio Mathematica}, 150(12):2054–2072, 2014.

\bibitem[PS16]{PS16}
Yu. Prokhorov and C. Shramov.
\newblock Jordan property for Cremona groups.
\newblock {\em Amer. J. Math.}, 138(2):403–418, 2016.

\bibitem[PS18]{PS18}
Yu. Prokhorov and C. Shramov.
\newblock Finite groups of birational selfmaps of threefolds. 
\newblock {\em Math. Res. Lett.} 25 (2018), no. 3, 957–972. 

\bibitem[PS20]{PS20}
Yu. Prokhorov and C. Shramov.
\newblock Finite groups of bimeromorphic selfmaps of uniruled Kähler threefolds.
\newblock {\em Izv. Ross. Akad. Nauk Ser. Mat.}, 84(5):169–196, 2020.

\bibitem[PS21a]{PS21a}
Yu. Prokhorov and C. Shramov.
\newblock Automorphism groups of compact complex surfaces.
\newblock {\em Int. Math. Res. Notices}, 2021(14):10490–10520, 2021.

\bibitem[PS21b]{PS21b}
Yu. Prokhorov and C. Shramov.
\newblock Finite groups of bimeromorphic selfmaps of non-uniruled Kähler threefolds.
\newblock arXiv:2110.05825

\bibitem[Ser07]{Ser07}
J.-P. Serre.
\newblock Bounds for the orders of the finite subgroups of $G(k)$.
\newblock In {\em Group representation theory}, pages
405–450. EPFL Press, Lausanne, 2007.

\bibitem[Zar14]{Zar14}
Yu. Zarhin.
\newblock Theta groups and products of abelian and rational varieties.
\newblock {\em Proc. Edinb. Math. Soc.} (2) 57 (2014), no. 1, 299--304.

\end{thebibliography}
\end{document}